\documentclass[12pt]{article}

\usepackage{amsfonts, amssymb, amsmath, amsthm, epsf}

\numberwithin{figure}{section}

\theoremstyle{plain}
\newtheorem{theorem}{Theorem}[section]
\newtheorem{lemma}{Lemma}[section]
\newtheorem{maximum}{Maximum Principle}[section]
\newtheorem{corollary}{Corollary}[section]
\newtheorem{condition}{Condition}[section]

\theoremstyle{definition}

\newtheorem{remark}{Remark}[section]

\numberwithin{equation}{section}

\newcommand{\loc}{{\rm loc}}
\newcommand{\ind}{{\rm ind\,}}
\newcommand{\supp}{{\rm supp\,}}
\newcommand{\spt}{{\rm spt\,}}
\newcommand{\Dom}{{\rm D}}

\renewcommand{\ker}{{\rm ker\,}}
\renewcommand{\dim}{{\rm dim\,}}

\newcommand{\dist}{{\rm dist}}
\renewcommand{\phi}{{\varphi}}

\newcommand{\cK}{{\mathcal K}}
\newcommand{\cM}{{\mathcal M}}
\newcommand{\cN}{{\mathcal N}}

\newcommand{\cO}{{\mathcal O}}
\newcommand{\cC}{{\mathcal C}}

\newcommand{\pG}{{\partial G}}

\newcommand{\oG}{{\overline G}}

\newcommand{\bP}{{\mathbf P}}
\newcommand{\bB}{{\mathbf B}}

\newcommand{\bS}{{\mathbf S}}
\newcommand{\bI}{{\mathbf I}}
\newcommand{\bT}{{\mathbf T}}

\newcommand{\bbR}{{\mathbb R}}

\title{Bounded perturbations of two-dimensional diffusion processes with
nonlocal conditions near the boundary}
\author{Pavel~Gurevich\thanks{Supported by the Russian Foundation for Basic Research (project
No.~04-01-00256) and the Alexander von Humboldt Foundation.}}

\date{}

\begin{document}

\maketitle

\begin{abstract}
We study the existence of Feller semigroups arising in the theory
of multidimensional diffusion processes. We study bounded
perturbations of elliptic operators with boundary conditions
containing an integral over the closure of the domain with respect
to a nonnegative Borel measure without assuming that the measure
is small. We state sufficient conditions on the measure
guaranteeing that the corresponding nonlocal operator is the
generator of a Feller semigroup.
\end{abstract}

\section{Introduction and Preliminaries}

In~\cite{Feller1,Feller2}, Feller investigated a general form of a generator of a
strongly continuous contractive nonnegative semigroup of operators acting between the
spaces of continuous functions on an interval, a half-line, or the whole line. Such a
semigroup corresponds to the one-dimensional diffusion process and is now called the
Feller semigroup. In the multidimensional case, the general form of a generator of a
Feller semigroup has been obtained by Ventsel~\cite{Ventsel}. Under some regularity
assumptions concerning the Markov process, he proved that the generator of the
corresponding Feller semigroup is an elliptic differential operator of second order
(possibly with degeneration) whose domain of definition consists of continuous (once or
twice continuously differentiable, depending on the process) functions satisfying
nonlocal conditions which involve an  integral of a function  over the closure of the
region with respect to a nonnegative Borel measure $\mu(y,d\eta)$. The inverse question
remains open: given an elliptic integro-differential operator whose domain of definition
is described by nonlocal boundary conditions, whether or not this operator (or its
closure) is a generator of a Feller semigroup.

One distinguishes two classes of nonlocal boundary conditions: the so-called {\it
transversal} and {\it nontransversal} ones. The order of nonlocal terms is less than the
order of local terms in the transversal case, and these orders coincide in the
nontransversal case (see, e.g.,~\cite{Taira3} for details and probabilistic
interpretation). The transversal case was studied in~\cite{SU, BCP, Taira1, Taira3,
Ishikawa, GalSkJDE}. The more difficult nontransversal nonlocal conditions are dealt with
in~\cite{SkDAN89,SkRJMP95,GalSkMs, GalSkJDE}.

It was assumed in~\cite{SkDAN89,SkRJMP95} that the coefficients at nonlocal terms
decrease as the argument tends to the boundary. In~\cite{GalSkMs,GalSkJDE}, the authors
considered nonlocal conditions with the coefficients that are less than one. This allowed
them to regard (after reduction to the boundary) the nonlocal problem as a perturbation
of the ``local'' Dirichlet problem.

In this paper, we  consider nontransversal nonlocal conditions on the boundary of a plane
domain $G$, admitting  ``limit case'' where the measure $\mu(y,\oG)$, after some
normalization, may equal one (it cannot be greater than one~\cite{Ventsel}). We assume
that if the support of the measure $\mu(y,d\eta)$ is ``close'' to the point $y$ for some
$y\in\pG$ and  $\mu(y,\oG)=1$, then the measure  $\mu(y,d\eta)$ is atomic.

Based on the Hille--Iosida theorem and on the solvability of elliptic equations with
nonlocal terms supported near the boundary~\cite{GurMIAN2007}, we provide a class of
Borel measures $\mu(y,d\eta)$ for which the corresponding nonlocal operator  is a
generator of a Feller semigroup.

In the conclusion of this section, we remind the notion of a
Feller semigroup and its generator and formulate a version of the
Hille--Iosida theorem adapted for our purposes.

\bigskip

Let $G\subset\bbR^2$ be a bounded domain with piecewise smooth
boundary $\pG$, and let $X$ be a closed subspace in $C(\oG)$
containing at least one nontrivial nonnegative function.

A strongly continuous semigroup of operators $\bT_t:X\to X$ is called a {\it Feller
semigroup}   {\it on} $X$ if it satisfies the following conditions: 1. $\|\bT_t\|\le 1$,
$t\ge0$; 2.  $\bT_t u\ge0$ for all $t\ge0$ and $u\in X$, $u\ge0$.

A linear operator $\bP:\Dom(\bP)\subset X\to X$ is called the
({\it infinitesimal}) {\it generator} of a strongly continuous
semigroup $\{\bT_t\}$ if $ \bP u=\lim\limits_{t\to +0}{(\bT
u-u)}/{t},\  \Dom(\bP)=\{u\in X:  \text{the limit  exists in }
X\}. $

\begin{theorem}[the Hille--Iosida theorem, see Theorem~9.3.1 in~\cite{Taira1}]\label{thHI}
\begin{enumerate}
\item
Let  $\bP:\Dom(\bP)\subset X\to X$  be a generator of a Feller
semigroup on $X$. Then the following assertions are true.
\begin{enumerate}
\item[$(a)$]
The domain $\Dom(\bP)$ is dense in $X$.
\item[$(b)$]
For each $q>0$ the operator $q\bI-\bP$ has the bounded inverse
$(q\bI-\bP)^{-1}:X\to X$ and $\|(q\bI-\bP)^{-1}\|\le 1/q$.
\item[$(c)$]
The operator $(q\bI-\bP)^{-1}:X\to X$, $q>0$, is nonnegative.
\end{enumerate}
\item
Conversely, if $\bP$ is a linear operator from $X$ to $X$
satisfying condition $(a)$ and there is a constant $q_0\ge 0$ such
that conditions $(b)$ and $(c)$ hold for $q>q_0$, then $\bP$ is
the generator of a certain Feller semigroup  on $X$, which is
uniquely determined by $\bP$.
\end{enumerate}
\end{theorem}

\section{Nonlocal Conditions near the Conjugation Points}\label{subsectStatement}
Consider a set ${\cK}\subset\partial G$ consisting of finitely
many points. Let $\partial G\setminus{\mathcal
K}=\bigcup\limits_{i=1}^{N}\Gamma_i$, where $\Gamma_i$ are open
(in the topology of $\partial G$) $C^\infty$ curves. Assume that
the domain $G$ is a plane angle in some neighborhood of each point
$g\in{\mathcal K}$.

For an integer $k\ge0$, denote by $W_2^k(G)$ the usual Sobolev
space.  Denote by $W^k_{2,\loc}(G)$ ($k\ge0$ is an integer) the
set of functions $u$ such that $u\in W_2^k(G')$ for any domain
$G'$, $\overline{G'}\subset G$.

Consider the differential operator $$
P_0u=\sum\limits_{j,k=1}^{2}p_{jk}(y)u_{y_jy_k}(y)+
\sum\limits_{j=1}^2p_j(y)u_{y_j}(y)+p_0(y)u(y), $$ where
$p_{jk},p_j\in C^\infty(\bbR^2)$ are real-valued functions and
$p_{jk}=p_{kj}$, $j,k=1,2$.

\begin{condition}\label{cond1.1}
1. There is a constant $c>0$ such that
$\sum\limits_{j,k=1}^{2}p_{jk}(y)\xi_j\xi_k\ge c|\xi|^2$ for
$y\in\overline{G}$ and $\xi=(\xi_1,\xi_2)\in\bbR^2.$ 2.
 $p_0(y)\le0$ for
$y\in\overline{G}$.
\end{condition}

In the sequel, we will use the following version of the well-known
maximum principle.

\begin{maximum}[see Theorem 9.6 in~\cite{GilbTrud}]\label{mp2}
Let $D\subset\bbR^2$ be a bounded or unbounded domain, and let
Condition~$\ref{cond1.1}$ hold with $G$ replaced by $D$. If a
function $u\in C(D)$ achieves its positive maximum at a point
$y^0\in D$ and\footnote{Here and below the operator $P_0$ acts in
the sense of distributions.} $P_0u\in C(D)$, then $P_0
u(y^0)\le0$.
\end{maximum}

Introduce the operators corresponding to nonlocal terms supported
near the set $\mathcal K$. For any  set $\mathcal M$, we denote
its $\varepsilon$-neighborhood by $\mathcal
O_{\varepsilon}(\mathcal M)$. Let $\Omega_{is}$ ($i=1, \dots, N;$
$s=1, \dots, S_i$) be $C^\infty$ diffeomorphisms taking some
neighborhood ${\mathcal O}_i$ of the curve
$\overline{\Gamma_i\cap\mathcal O_{{\varepsilon}}(\mathcal K)}$ to
the set $\Omega_{is}({\mathcal O}_i)$ in such a way that
$\Omega_{is}(\Gamma_i\cap\mathcal O_{{\varepsilon}}(\mathcal
K))\subset G$ and $ \Omega_{is}(g)\in\mathcal K$ for $
g\in\overline{\Gamma_i}\cap\mathcal K. $ Thus, the transformations
$\Omega_{is}$ take the curves $\Gamma_i\cap\mathcal
O_{{\varepsilon}}(\mathcal K)$ strictly inside the domain $G$ and
the set of their end points $\overline{\Gamma_i}\cap\mathcal K$ to
itself.

Let us specify the structure of the transformations $\Omega_{is}$
near the set $\mathcal K$. Denote by $\Omega_{is}^{+1}$ the
transformation $\Omega_{is}:{\mathcal O}_i\to\Omega_{is}({\mathcal
O}_i)$ and by $\Omega_{is}^{-1}:\Omega_{is}({\mathcal
O}_i)\to{\mathcal O}_i$ the inverse transformation. The set of
points
$\Omega_{i_qs_q}^{\pm1}(\dots\Omega_{i_1s_1}^{\pm1}(g))\in{\mathcal
K}$ ($1\le s_j\le S_{i_j},\ j=1, \dots, q$) is said to be an {\em
orbit} of the point $g\in{\mathcal K}$. In other words, the orbit
of a point $g$ is formed by the points (of the set $\mathcal K$)
that can be obtained by consecutively applying the transformations
$\Omega_{i_js_j}^{\pm1}$ to the point $g$. The set $\mathcal K$
consists of finitely many disjoint orbits, which we denote by
$\mathcal K_\nu$.

Take a sufficiently small number $\varepsilon>0$ such that there
exist neighborhoods $\mathcal O_{\varepsilon_1}(g_j)$, $ \mathcal
O_{\varepsilon_1}(g_j)\supset\mathcal O_{\varepsilon}(g_j) $,
satisfying the following conditions: 1. the domain $G$ is a plane
angle in the neighborhood $\mathcal O_{\varepsilon_1}(g_j)$;  2.
$\overline{\mathcal O_{\varepsilon_1}(g)}\cap\overline{\mathcal
O_{\varepsilon_1}(h)}=\varnothing$ for any $g,h\in\mathcal K$,
$g\ne h$; 3. if $g_j\in\overline{\Gamma_i}$ and
$\Omega_{is}(g_j)=g_k,$ then ${\mathcal
O}_{\varepsilon}(g_j)\subset\mathcal
 O_i$ and
 $\Omega_{is}\big({\mathcal
O}_{\varepsilon}(g_j)\big)\subset{\mathcal
O}_{\varepsilon_1}(g_k).$

For each point $g_j\in\overline{\Gamma_i}\cap\mathcal K_\nu$, we
fix a linear transformation $Y_j: y\mapsto y'(g_j)$ (the
composition of the shift by the vector $-\overrightarrow{Og_j}$
and rotation) mapping the point $g_j$ to the origin in such a way
that $ Y_j({\mathcal O}_{\varepsilon_1}(g_j))={\mathcal
O}_{\varepsilon_1}(0),\ Y_j(G\cap{\mathcal
O}_{\varepsilon_1}(g_j))=K_j\cap{\mathcal O}_{\varepsilon_1}(0), $
$ Y_j(\Gamma_i\cap{\mathcal
O}_{\varepsilon_1}(g_j))=\gamma_{j\sigma}\cap{\mathcal
O}_{\varepsilon_1}(0)\ (\sigma=1\ \text{or}\ 2), $ where $
 K_j$ is a plane angle of nonzero opening and $\gamma_{j\sigma}$ its sides.

\begin{condition}\label{condK1}
Let $g_j\in\overline{\Gamma_i}\cap\mathcal K_\nu$ and
$\Omega_{is}(g_j)=g_k\in\mathcal K_\nu;$ then the transformation $
Y_k\circ\Omega_{is}\circ Y_j^{-1}:{\mathcal
O}_{\varepsilon}(0)\to{\mathcal O}_{\varepsilon_1}(0) $ is the
composition of rotation and homothety centered at the origin.
\end{condition}

Introduce the nonlocal operators $\mathbf B_{i}$ by the formulas
\begin{equation}\label{eq3'}
 \mathbf B_{i}u=\sum\limits_{s=1}^{S_i}
b_{is}(y) u(\Omega_{is}(y)),\quad
   y\in\Gamma_i\cap\mathcal O_{\varepsilon}(\mathcal K),\qquad
\mathbf B_{i}u=0,\quad y\in\Gamma_i\setminus\mathcal
O_{\varepsilon}(\mathcal K),
\end{equation}
where $b_{is}\in C^\infty(\mathbb R^2)$ are real-valued functions,
$\supp b_{is}\subset\mathcal O_{{\varepsilon}}(\mathcal K)$.

\begin{condition}\label{cond1.2}
\begin{enumerate}
\item
$
b_{is}(y)\ge0,\qquad \sum\limits_{s=1}^{S_i} b_{is}(y)\le 1,\qquad
 y\in\overline{\Gamma_i};
$
\item
$
 \sum\limits_{s=1}^{S_i}
b_{is}(g)+\sum\limits_{s=1}^{S_j} b_{js}(g)<2,\quad
g\in\overline{\Gamma_i}\cap\overline{\Gamma_j}\subset\cK,\qquad\text{if}\
 i\ne j\ \text{and}\
 \overline{\Gamma_i}\cap\overline{\Gamma_j}\ne\varnothing.
$
\end{enumerate}
\end{condition}

Now we formulate some auxiliary results to be used in the
next sections.

For any closed sets $Q\subset\oG$ and $K\subset\oG$ such that
$Q\cap K\ne\varnothing$, we introduce the space
\begin{equation}\label{eqC_K}
C_K(Q)=\{u\in C(Q): u(y)=0,\ y\in Q\cap K\}
\end{equation}
with the maximum-norm. Consider the space of vector-valued
functions $ \cC_\cK(\pG)=\prod\limits_{i=1}^N
C_\cK(\overline{\Gamma_i}) $ with the norm $
\|\psi\|_{\cC_\cK(\pG)}=
\max\limits_{i=1,\dots,N}\max\limits_{y\in\overline{\Gamma_i}}\|\psi_i\|_{C(\overline{\Gamma_i})},
$ where $\psi=\{\psi_i\}$, $\psi_i\in C_\cK(\overline{\Gamma_i})$.

Consider the problem
\begin{equation}\label{eq47-48}
P_0u-q u =f_0(y), \  y\in G;\qquad u|_{\Gamma_i}-\bB_i
u=\psi_i(y), \ y\in\Gamma_i,\ i=1,\dots,N.
\end{equation}

\begin{theorem}[see Theorem 4.1 in~\cite{GurMIAN2007}]\label{th1-2}
Let Conditions~$\ref{cond1.1}$--$\ref{cond1.2}$ be fulfilled. Then there is a number
$q_1>0$ such that, for any $f_0\in C(\oG)$, $\psi=\{\psi_i\}\in \cC_\cK(\pG)$, and $q\ge
q_1$, there exists a unique solution $u\in C_\cK(\overline G)\cap W_{2,\loc}^2(G)$ of
problem~\eqref{eq47-48}. Furthermore, if $f_0=0$, then $u\in C_\cK(\overline G)\cap
C^\infty(G)$ and the following estimate holds{\rm:}
\begin{equation}\label{eq49}
\|u\|_{C_\cK(\overline G)}\le c_1\|\psi\|_{\cC_\cK(\pG)},
\end{equation}
where $c_1>0$ does not depend on  $\psi$ and $q$.
\end{theorem}

Let  $u\in C^\infty(G)\cap C_\cK(\oG)$ be a solution of
problem~\eqref{eq47-48} with $f_0=0$ and
$\psi=\{\psi_i\}\in\cC_\cK(\pG)$. Denote $u=\bS_q\psi$. By
Theorem~\ref{th1-2}, the operator
$$
\bS_q: \cC_\cK(\pG)\to C_\cK(\overline G),\qquad q\ge q_1,
$$
is bounded and $\|\bS_q\|\le c_1$, where $c_1>0$ does not depend
on $q$.

\begin{lemma}\label{l4}
Let Conditions~$\ref{cond1.1}$--$\ref{cond1.2}$ hold, let $Q_1$
and $Q_2$ be closed sets such that $Q_1\subset\pG$,
$Q_2\subset\overline G$, and $Q_1\cap Q_2=\varnothing$, and let
$q\ge q_1$. Then the inequality
$$
\|\bS_q\psi\|_{C(Q_2)}\le\dfrac{c_2}{q}\|\psi\|_{\cC_\cK(\pG)},\qquad
q\ge q_1,
$$
holds for any $\psi\in\cC_\cK(\pG)$ such that
$\supp(\bS_q\psi)|_{\pG}\subset Q_1;$ here $c_2>0$ does not depend
on $\psi$ and $q$.
\end{lemma}
\begin{proof}
Using\footnote{It is supposed in Lemma 1.3 in~\cite{GalSkMs} that
the boundary of domain is infinitely smooth. This assumption is
needed to prove the existence of classical solution for elliptic
equations with nonhomogeneous boundary condition. However, this
assumption is needless for the validity of the first inequality in
\eqref{eq55_0}, provided that the solution exists.} Lemma 1.3 in
\cite{GalSkMs} and Theorem~\ref{th1-2}, we obtain
\begin{equation}\label{eq55_0}
\|\bS_q\psi\|_{C(Q_2)}\le \dfrac{k}{q}
\|(\bS_q\psi)|_{\pG}\|_{C(\pG)}\le
\dfrac{k}{q}\|\bS_q\psi\|_{C(\overline
G)}\le\dfrac{kc_1}{q}\|\psi\|_{\cC_\cK(\pG)},\qquad q\ge q_1,
\end{equation}
where the number $q_1$ defined in Theorem~\ref{th1-2} is assumed
to be large enough so that Lemma 1.3 in~\cite{GalSkMs} be valid
for $q\ge q_1$; the number $k=k(q_1)$ does not depend on $\psi$
and $q$.
\end{proof}

\begin{lemma}\label{l5}
Let Conditions~$\ref{cond1.1}$--$\ref{cond1.2}$ hold,  let $Q_1$
and $Q_2$ be the same sets as in Lemma $\ref{l4}$, and let $q\ge
q_1$. We additionally suppose that $Q_2\cap\cK=\varnothing$. Then
the inequality
$$
\|\bS_q\psi\|_{C(Q_2)}\le\dfrac{c_3}{q}\|\psi\|_{\cC_\cK(Q_1)},\qquad
q\ge q_1,
$$
holds for any $\psi\in\cC_\cK(\pG)$ such that $\supp \psi\subset
Q_1;$ here $c_3>0$ does not depend on $\psi$ and $q$.
\end{lemma}
\begin{proof}
1. Consider a number $\sigma>0$ such that
\begin{equation}\label{eq55_2}
\dist(Q_1,Q_2)>3\sigma,\qquad \dist(\cK, Q_2)>3\sigma.
\end{equation}
Introduce  a function $\xi\in C^\infty(\bbR^2)$ such that $0\le
\xi(y)\le1$, $\xi(y)=1$ for $\dist (y,Q_2)\le\sigma$, and
$\xi(y)=0$ for $\dist (y,Q_2)\ge 2\sigma$.

Consider the auxiliary problem
\begin{equation}\label{eq55_3-4}
P_0v-q v=0,\  y\in G;\qquad v(y)=\xi(y)u(y),\ y\in\pG,
\end{equation}
where $u=\bS_q\psi\in C_\cK(\overline G)$.  Applying
Theorem~\ref{th1-2} with $\bB_i =0$, we see that there is a unique
solution $v\in C^\infty(G)\cap C(\overline G)$ of problem
\eqref{eq55_3-4}. If follows from  Maximum Principle~\ref{mp2} and
from the definition of the function $\xi$ that
\begin{equation}\label{eq55_5}
\|v\|_{C(\overline G)}\le \|\xi
u\|_{C(\pG)}\le\max\limits_{i=1,\dots,N}
\|u|_{Q_{2,2\sigma}\cap\overline{\Gamma_i}}\|_{C(Q_{2,2\sigma}\cap\overline{\Gamma_i})},
\end{equation}
where $Q_{2,2\sigma}=\{y\in\pG: \dist(y,Q_2)\le 2\sigma\}$.

Since $\supp \psi\cap Q_{2,2\sigma}=\varnothing$, it follows that
\begin{equation}\label{eq55_6}
u-\bB_i  u=0,\qquad y\in Q_{2,2\sigma}\cap\overline{\Gamma_i}.
\end{equation}
Taking into account that $\bB_i u=0$ for
$y\notin\cO_\varepsilon(\cK)$, we deduce from~\eqref{eq55_6} that
\begin{equation}\label{eq55_7}
u(y)=0,\qquad y\in [Q_{2,2\sigma}\cap\overline{\Gamma_i}]\setminus
\cO_\varepsilon(\cK).
\end{equation}

Using~\eqref{eq55_5}--\eqref{eq55_7}, the definition of the
operators $\bB_i $, and Condition~\ref{cond1.2}, we obtain
\begin{equation}\label{eq55_8}
\begin{aligned}
\|v\|_{C(\overline G)}&\le \max\limits_{i=1,\dots,N}
\|u|_{Q_{2,2\sigma}\cap\overline{\Gamma_i}\cap\overline{\cO_\varepsilon(\cK)}}\|_{
C(Q_{2,2\sigma}\cap\overline{\Gamma_i}\cap\overline{\cO_\varepsilon(\cK)})}\\
&\le \max\limits_{i=1,\dots,N}\max\limits_{s=1,\dots,S_i}
\|u|_{\Omega_{is}(Q_{2,2\sigma}\cap\overline{\Gamma_i}\cap\overline{\cO_\varepsilon(\cK)})}\|_{
C(\Omega_{is}(Q_{2,2\sigma}\cap\overline{\Gamma_i}\cap\overline{\cO_\varepsilon(\cK)}))}.
\end{aligned}
\end{equation}

Since $Q_{2,2\sigma}\cap\cK=\varnothing$ (see~\eqref{eq55_2}), it
follows from the definition of the transformations $\Omega_{is}$
that
$$
\Omega_{is}(Q_{2,2\sigma}\cap\overline{\Gamma_i}\cap\overline{\cO_\varepsilon(\cK)}))\subset
G.
$$
Therefore, using inequality~\eqref{eq55_8} and Lemma~\ref{l4} with
$Q_1$ and $Q_2$ replaced by $\pG$ and
$\Omega_{is}(Q_{2,2\sigma}\cap\overline{\Gamma_i}\cap\overline{\cO_\varepsilon(\cK)}))$,
we have
\begin{equation}\label{eq55_9}
\|v\|_{C(\overline G)}\le \dfrac{c_2}{q}\|\psi\|_{\cC_\cK(\pG)}.
\end{equation}

2. Set $w=u-v$. Clearly, the function $w$ satisfies the relations
$$
P_0w-q w =0,\  y\in G;\qquad w(y)=u(y)-v(y) =0,\  y\in
Q_{2,\sigma}.
$$
Applying Lemma~\ref{l4} with $\overline{\pG\setminus
Q_{2,\sigma}}$ substituted for $Q_1$ and $\bB_i=0$ and taking into
account that $w|_{\pG}=(1-\xi)  u|_{\pG}$, we obtain
$$
\|w\|_{C(Q_2)}\le
\dfrac{c_2}{q}\|w|_{\pG}\|_{C(\pG)}\le\dfrac{c_2}{q}\|u\|_{C(\overline
G)}.
$$
The latter inequality and Theorem~\ref{th1-2} imply
$$
\|w\|_{C(Q_2)}\le \dfrac{c_2c_1}{q}\|\psi\|_{\cC_\cK(\pG)}.
$$
Combining this estimate with~\eqref{eq55_9}, we complete the
proof.
\end{proof}

\section{Bounded Perturbations of Elliptic Operators and Their
Properties}\label{subsectBoundedHypoth}

Introduce a linear operator $P_1$ satisfying the following
condition.
\begin{condition}\label{cond2.1'}
The operator $P_1: C(\overline G)\to C(\overline G)$ is bounded,
and $P_1 u(y^0)\le 0$ whenever $u\in C(\overline G)$ achieves its
positive maximum at the point $y^0\in G$.
\end{condition}

The operator $P_1$ will play the role of a bounded perturbation
for unbounded elliptic operators in the spaces of continuous
functions (cf.~\cite{GalSkMs, GalSkJDE}).

The following result is a consequence of Conditions~\ref{cond1.1}
and~\ref{cond2.1'} and Maximum Principle~\ref{mp2}.

\begin{lemma}\label{l2.1}
Let Conditions $\ref{cond1.1}$ and $\ref{cond2.1'}$ hold. If a
function $u\in C(\oG)$ achieves its positive maximum at a point
$y^0\in G$ and $P_0u\in C(G)$, then $P_0u(y^0)+P_1 u(y^0)\le0$.
\end{lemma}

In this paper, we consider the following nonlocal conditions in
the
 {\it nontransversal} case:
\begin{equation}\label{eq56}
b(y)u(y)+\int\limits_{\oG}[u(y)-u(\eta)]\mu(y,d\eta)=0,\qquad
y\in\pG,
\end{equation}
where $b(y)\ge0$ and $\mu(y,\cdot)$ is a nonnegative Borel measure on
$\oG$.

Set $
 \cN=\{y\in\pG: \mu(y,\oG)=0\}$ and $\cM=\pG\setminus \cN.$
Assume that $\cN$ and $\cM$ are Borel sets.

\begin{condition}\label{cond2.1''}
$\cK\subset \cN$.
\end{condition}

Introduce the function $ b_0(y)=b(y)+\mu(y,\oG). $

\begin{condition}\label{cond2.2}
$b_0(y)>0$ for $y\in\pG$.
\end{condition}

Conditions~\ref{cond2.1''} and~\ref{cond2.2} imply that relation
\eqref{eq56} can be written as follows:
\begin{equation}\label{eq57}
u(y)-\int\limits_\oG u(\eta)\mu_i(y,d\eta)=0,\ y\in\Gamma_i;\qquad
u(y) =0,\  y\in\cK,
\end{equation}
where $ \mu_i(y,\cdot)=\dfrac{\mu(y,\cdot)}{b_0(y)},\
y\in\Gamma_i. $ By the definition of the function $b_0(y)$, we
have
\begin{equation}\label{eq59}
\mu_i(y,\oG)\le 1,\qquad y\in\Gamma_i.
\end{equation}

For any set $Q$, we denote by $\chi_Q(y)$ the function equal to
one on $Q$ and vanishing on $\bbR^2\setminus Q$.

Let $b_{is}(y)$ and $\Omega_{is}$ be the same  as above. We
introduce the measures $\delta_{is}$ as follows:
$$
\delta_{is}(y,Q)=\left\{
\begin{aligned}
&b_{is}(y)\chi_Q(\Omega_{is}(y)),& &y\in\Gamma_i\cap\cO_\varepsilon(\cK),\\
&0,& &y\in\Gamma_i\setminus\cO_\varepsilon(\cK),
\end{aligned}\right.
$$
for any Borel set $Q$.

 We study those  measures $\mu_i(y,\cdot)$ which
can be represented in the form
\begin{equation}\label{eq61}
\mu_i(y,\cdot)=\sum\limits_{s=1}^{S_i}\delta_{is}(y,\cdot)+\alpha_i(y,\cdot)+\beta_i(y,\cdot),\qquad
y\in\Gamma_i,
\end{equation}
where $\alpha_i(y,\cdot)$ and $\beta_i(y, \cdot)$ are nonnegative Borel measures to be
specified below (cf.~\cite{GalSkMs,GalSkJDE}).

For any Borel measure $\mu(y,\cdot)$, the closed set $
\spt\mu(y,\cdot)=\oG\setminus\bigcup\limits_{V\in T}\{V\in T:
\mu(y,V\cap\oG)=0\} $ (where~$T$ denotes the set of all open sets
in $\bbR^2$) is called the {\it support} of the measure
$\mu(y,\cdot)$.

\begin{condition}\label{cond2.3}
There exist numbers $\varkappa_1>\varkappa_2>0$ and $\sigma>0$
such that
\begin{enumerate}
\item
$\spt\alpha_i(y,\cdot)\subset\oG\setminus\cO_{\varkappa_1}(\cK)$
for $y\in\Gamma_i$,
\item
$\spt\alpha_i(y,\cdot)\subset\overline{G_\sigma}$ for
$y\in\Gamma_i\setminus\cO_{\varkappa_2}(\cK),$
\end{enumerate}
where
$\cO_{\varkappa_1}(\cK)=\{y\in\bbR^2:\dist(y,\cK)<\varkappa_1\}$
and $G_\sigma=\{y\in G:\dist(y,\pG)<\sigma\}.$
\end{condition}

\begin{condition}\label{cond2.4}
 $\beta_i(y,\cM)<1$ for $y\in\Gamma_i\cap\cM$, $i=1,\dots,N$.
\end{condition}

\begin{remark}
Condition~\ref{cond2.4} is weaker than (analogous) Condition 2.2
in~\cite{GalSkMs} or Condition 3.2 in~\cite{GalSkJDE} because the
latter two  require that $\mu_i(y,\cM)<1$ for
$y\in\Gamma_i\cap\cM$.
\end{remark}

\begin{remark}
One can show that Conditions~\ref{cond2.2}--\ref{cond2.4} imply
that $ b(y)+\mu(y,\oG\setminus\{y\})>0,\ y\in\pG, $ i.e., the
boundary-value condition~\eqref{eq56} disappears nowhere on the
boundary.
\end{remark}

Using relations~\eqref{eq61}, we write nonlocal conditions
\eqref{eq57} in the form
\begin{equation}\label{eq63}
u(y)-\bB_i u(y)-\bB_{\alpha i}u(y)-\bB_{\beta i}u(y) =0,\
y\in\Gamma_i;\qquad u(y) =0,\ y\in\cK,
\end{equation}
where the operators $\bB_i $ are given by~\eqref{eq3'} and
$$
\bB_{\alpha i}u(y)=\int\limits_\oG u(\eta)\alpha_i(y,
d\eta),\qquad \bB_{\beta i}u(y)=\int\limits_\oG u(\eta)\beta_i(y,
d\eta),\qquad y\in\Gamma_i.
$$

Introduce the space\footnote{Clearly, nonlocal conditions
\eqref{eq56} in the definition of the space $C_B(\oG)$ can be
replaced by conditions~\eqref{eq57} or~\eqref{eq63}.} $
C_B(\oG)=\{u\in C(\oG): u\ \text{satisfy nonlocal conditions
\eqref{eq56}}\}. $

It follows from the definition of the space $C_B(\oG)$ and from
Condition~\ref{cond2.1''} that\footnote{The  spaces $C_\cN(\cdot)$
and $C_\cK(\cdot)$ are given in \eqref{eqC_K}.}
\begin{equation}\label{eqBNK}
C_B(\oG)\subset C_\cN(\oG)\subset C_\cK(G).
\end{equation}

\begin{lemma}\label{l2.3}
Let Conditions $\ref{cond1.1}$--$\ref{cond1.2}$ and
$\ref{cond2.1'}$--$\ref{cond2.4}$ hold. Let a function $u\in
C_B(\oG)$ achieve its positive maximum at a point $y^0\in\overline
G$ and $P_0u\in C(G)$. Then there is a point $y^1\in G$ such that
$u(y^1)=u(y^0)$ and $P_0u(y^1)+P_1u(y^1)\le 0$.
\end{lemma}
\begin{proof}
1. If $y^0\in G$, then the conclusion of the lemma follows from
Lemma~\ref{l2.1}. Let $y^0\in\pG$. Suppose that the lemma is not
true, i.e., $u(y^0)>u(y)$ for $y\in G$.

Since $u(y^0)>0$ and $u\in C_B(\oG)\subset C_\cN(\oG)$, it follows
that $y^0\in \cM$. Let $y^0\in\Gamma_i\cap \cM$ for some $i$. If
$\mu_i(y^0,G)>0$, then, taking into account~\eqref{eq59}, we have
$$
u(y^0)-\int\limits_\oG u(\eta)\mu_i(y^0,d\eta)\ge \int\limits_G
[u(y^0)-u(\eta)]\mu_i(y^0,d\eta)>0,
$$
which contradicts~\eqref{eq57}. Therefore,
$\spt\mu_i(y^0,\cdot)\subset\pG$. It follows from this relation,
from~\eqref{eq61}, and from Condition~\ref{cond2.3} (part 1) that
\begin{equation}\label{eq65}
b_{is}(y^0)=0,\qquad
\spt\alpha_i(y^0,\cdot)\subset\pG\setminus\cO_{\varkappa_1}(\cK),\qquad
\spt\beta_i(y^0,\cdot)\subset\pG.
\end{equation}

2. Suppose that
$\alpha_i(y^0,\pG\setminus\cO_{\varkappa_1}(\cK))=0$. In this
case, due to~\eqref{eq65},
\begin{equation}\label{eq67}
\alpha_i(y^0,\oG)=0.
\end{equation}
Now it follows from~\eqref{eq61},~\eqref{eq65},~\eqref{eq67} and
from Condition~\ref{cond2.4} that
$$
\mu_i(y^0,\cdot)=\beta_i(y^0,\cdot),\qquad
\spt\beta_i(y^0,\cdot)\subset\pG,\qquad \beta_i(y^0,\cM)<1.
$$
Hence, the following inequalities hold for $u\in C_B(\oG)\subset
C_\cN(\oG)$:
$$
u(y^0)-\int\limits_\oG u(\eta)\mu_i(y^0,d\eta)=
u(y^0)-\int\limits_\cM u(\eta)\beta_i(y^0,d\eta) \ge
u(y^0)-u(y^0)\beta_i(y^0,\cM)>0,
$$
which contradicts~\eqref{eq57}.

This contradiction shows that
$\alpha_i(y^0,\pG\setminus\cO_{\varkappa_1}(\cK))>0$. Therefore,
taking into account  Condition~\ref{cond2.3} (part 2), we have
$y^0\in\cO_{\varkappa_2}(\cK)$.

3. We claim that there is a point
\begin{equation}\label{eq66}
y'\in\pG\setminus\cO_{\varkappa_1}(\cK)
\end{equation}
such that $u(y')=u(y^0)$. Indeed, assume the contrary:
$u(y^0)>u(y)$ for $y\in\pG\setminus\cO_{\varkappa_1}(\cK)$. Then,
using~\eqref{eq59},~\eqref{eq61}, and~\eqref{eq65}, we obtain
\begin{equation}\label{eq66'}
u(y^0)-\int\limits_\oG u(\eta)\mu_i(y^0,d\eta)\ge \int\limits_\oG
[u(y^0)-u(\eta)]\mu_i(y^0,d\eta)\ge
\int\limits_{\pG\setminus\cO_{\varkappa_1}(\cK)}
[u(y^0)-u(\eta)]\alpha_i(y^0,d\eta)>0
\end{equation}
because $\alpha_i(y^0,\pG\setminus\cO_{\varkappa_1}(\cK))>0$.
Inequality~\eqref{eq66'} contradicts~\eqref{eq57}. Therefore, the
function $u$ achieves its positive maximum at some point
$y'\in\pG\setminus\cO_{\varkappa_1}(\cK)$. Repeating the arguments
of items 1 and 2 of this proof yields
$y'\in\cO_{\varkappa_2}(\cK)$, which contradicts~\eqref{eq66}.

Thus, we have proved that there is a point $y^1\in G$ such that
$u(y^1)=u(y^0)$. Applying Lemma~\ref{l2.1}, we obtain
$P_0u(y^1)+P_1u(y^1)\le 0$.
\end{proof}

\begin{corollary}\label{cor2.1}
Let Conditions $\ref{cond1.1}$--$\ref{cond1.2}$ and
$\ref{cond2.1'}$--$\ref{cond2.4}$ hold. Let $u\in C_B(\oG)$ be a
solution of the equation $$qu(y)-P_0u(y)-P_1u(y)=f_0(y),\quad y\in
G,$$ where $q>0$ and $f_0\in C(\oG)$. Then
\begin{equation}\label{eqcor2.1}
\|u\|_{C(\oG)}\le\dfrac{1}{q}\|f_0\|_{C(\oG)}.
\end{equation}
\end{corollary}
\begin{proof}
Let $\max\limits_{y\in\oG}|u(y)|=u(y^0)>0$ for some $y^0\in \oG$.
In this case, by Lemma~\ref{l2.3}, there is a point $y^1\in G$
such that $u(y^1)=u(y^0)$ and $P_0u(y^1)+P_1u(y^1)\le 0$.
Therefore,
$$
\|u\|_{C(\oG)}=u(y^0)=u(y^1)=\dfrac{1}{q}(P_0u(y^1)+P_1u(y^1)+f_0(y^1))\le
\dfrac{1}{q}\|f_0\|_{C(\oG)}.
$$
\end{proof}

\section{Reduction to the Operator Equation on the Boundary}\label{subsectBoundedReduction}

In this section, we impose some additional restrictions on the
nonlocal operators, which allow us to reduce nonlocal elliptic
problems to operator equations on the boundary.

Note that if $u\in C_\cN(\oG)$, then $\bB_i  u$ is continuous on
$\Gamma_i$ and can be extended to a continuous function on
$\overline{\Gamma_i}$ (also denoted by $\bB_i  u$), which belongs
to  $C_\cN(\overline{\Gamma_i})$. We assume that the operators
$\bB_{\alpha i}$ and $\bB_{\beta i}$ possess the similar property.

\begin{condition}\label{cond2.5}
For any function $u\in C_\cN(\overline G)$, the functions
$\bB_{\alpha i}u$ and $\bB_{\beta i}u$ can be extended to
$\overline{\Gamma_i}$ in such a way that the extended functions
{\rm(}which we also denote by $\bB_{\alpha i}u$ and $\bB_{\beta
i}u$, respectively{\rm)} belong to $C_\cN(\overline{\Gamma_i})$.
\end{condition}

The next lemma directly follows from the definition of the
nonlocal operators.

\begin{lemma}\label{l2.2}
Let Conditions $\ref{condK1}$, $\ref{cond1.2}$, $\ref{cond2.1''}$,
$\ref{cond2.2}$, and $\ref{cond2.5}$ hold. Then the operators
$\bB_i , \bB_{\alpha i},\bB_{\beta i}: C_\cN(\oG)\to
C_\cN(\overline{\Gamma_i}) $ are bounded and
$$
\|\bB_i
u\|_{C_\cN(\overline{\Gamma_i})}\le\|u\|_{C_\cN(\oG)},\qquad
\|\bB_{\alpha
i}u\|_{C_\cN(\overline{\Gamma_i})}\le\|u\|_{C_\cN(\oG\setminus\cO_{\varkappa_1}(\cK))},\qquad
\|\bB_{\beta
i}u\|_{C_\cN(\overline{\Gamma_i})}\le\|u\|_{C_\cN(\oG)},
$$
$$
\|\bB_{\alpha i}u+\bB_{\beta i}u\|\le \|u\|_{C_\cN(\oG)},\qquad
\|\bB_i u+\bB_{\alpha i}u+\bB_{\beta i}u\|\le \|u\|_{C_\cN(\oG)}.
$$
\end{lemma}

Consider the space of vector-valued functions $
\cC_\cN(\pG)=\prod\limits_{i=1}^N C_\cN(\overline{\Gamma_i}) $
with the norm $ \|\psi\|_{\cC_\cN(\pG)}=
\max\limits_{i=1,\dots,N}\max\limits_{y\in\overline{\Gamma_i}}\|\psi_i\|_{C(\overline{\Gamma_i})}$,
$ \psi=\{\psi_i\},\ \psi_i\in C_\cN(\overline{\Gamma_i}) $.

Introduce the operators
\begin{equation}\label{eqBAlphaBeta}
\bB=\{\bB_i \}:C_\cN(\oG)\to\cC_\cN(\pG),\qquad
\bB_{\alpha\beta}=\{\bB_{\alpha i}+\bB_{\beta
i}\}:C_\cN(\oG)\to\cC_\cN(\pG).
\end{equation}

Using the operator $\bS_q$ defined in Sec. \ref{subsectStatement},
we introduce the bounded operator
\begin{equation}\label{eq68}
\bI-\bB_{\alpha\beta}\bS_q:\cC_\cN(\pG)\to\cC_\cN(\pG),\qquad q\ge
q_1.
\end{equation}
Since $\bS_q\psi\in C_\cN(\oG)$ for $\psi\in\cC_\cN(\pG)$, the
operator in~\eqref{eq68} is well defined.

Now we formulate sufficient conditions under which the bounded
operator $(\bI-\bB_{\alpha\beta}\bS_q)^{-1}:\cC_\cN(\pG)\to
\cC_\cN(\pG)$ exists.

We represent the measures $\beta_i(y,\cdot)$ in the form
\begin{equation}\label{eq73}
\beta_i(y,\cdot)=\beta_i^1(y,\cdot)+\beta_i^2(y,\cdot),
\end{equation}
where $\beta_i^1(y,\cdot)$ and $\beta_i^2(y,\cdot)$ are
nonnegative Borel measures. Let us specify them. For each $p>0$,
we consider the covering of the set $\overline{\cM}$ by the
$p$-neighborhoods of all its points. Denote some finite
subcovering by $\cM_p$. Since $\cM_p$ is a finite union of open
disks, it is an open Borel set. Now for each $p>0$, we consider a
cut-off function $\hat\zeta_p\in C^\infty(\bbR^2)$ such that
$0\le\hat\zeta_p(y)\le 1$, $\hat\zeta_p(y)=1$ for $y\in\cM_{p/2}$,
and $\hat\zeta_p(y)=0$ for $y\notin\cM_{p}$. Set
$\tilde\zeta_p=1-\hat\zeta_p$. Introduce the operators
$$
\hat\bB_{\beta i}^1
u(y)=\int\limits_{\oG}\hat\zeta_p(\eta)u(\eta)\beta_i^1(y,d\eta),\quad
 \tilde\bB_{\beta i}^1 u(y)=\int\limits_{\oG}\tilde\zeta_p(\eta)
u(\eta)\beta_i^1(y,d\eta),\quad  \bB_{\beta i}^2
u(y)=\int\limits_{\oG}u(\eta)\beta_i^2(y,d\eta).
$$

\begin{condition}\label{cond2.7}
The following assertions are true for $i=1,\dots,N${\rm:}
\begin{enumerate}
\item
the operators  $\hat\bB_{\beta i}^1,\tilde\bB_{\beta
i}^1:C_\cN(\oG)\to C_\cN(\overline{\Gamma_i})$ are bounded{\rm;}
\item there exists a number $p>0$ such that\footnote{Part 2 of Condition~\ref{cond2.7} may be replaced by the stronger assumption
$\|\hat\bB_{\beta i}^1\|\to 0$ as $p\to0$, which is easier to
verify in applications.}
$$
\|\hat\bB_{\beta i}^1\|< \left\{
\begin{aligned}
&\frac{1}{c_1} & &\text{if}\quad \alpha_j(y,\oG)=0\ \forall
y\in\Gamma_j,\
j=1,\dots,N,\\
&\frac{1}{c_1(1+c_1)} & &\text{otherwise},
\end{aligned}
\right.
$$
where $c_1$ is the constant occurring in Theorem {\rm\ref{th1-2}.}
\end{enumerate}
\end{condition}

\begin{remark}
The operators  $\hat\bB_{\beta i}^1,\tilde\bB_{\beta
i}^1:C_\cN(\oG)\to C_\cN(\overline{\Gamma_i})$ are bounded if and
only if the operator $\hat\bB_{\beta i}^1+\tilde\bB_{\beta
i}^1:C_\cN(\oG)\to C_\cN(\overline{\Gamma_i})$ is bounded. This
follows from the relations $\hat\bB_{\beta i}^1u=(\hat\bB_{\beta
i}^1+\tilde\bB_{\beta i}^1)(\hat\zeta_p u)$ and $\tilde\bB_{\beta
i}^1u=(\hat\bB_{\beta i}^1+\tilde\bB_{\beta i}^1)(\tilde\zeta_p
u)$ and from the continuity of the functions $\hat\zeta_p$ and
$\tilde\zeta_p$.
\end{remark}

\begin{condition}\label{cond2.8}
The operators $\bB_{\beta i}^2:C_\cN(\oG)\to
C_\cN(\overline{\Gamma_i})$, $i=1,\dots,N$, are compact.
\end{condition}

It follows from~\eqref{eq61} and~\eqref{eq73} that the measures
$\mu_i(y,\cdot)$ have the following representation:
\begin{equation*}
\mu_i(y,\cdot)=\sum\limits_{s=1}^{S_i}\delta_{is}(y,\cdot)+\alpha_i(y,\cdot)+
\beta_i^1(y,\cdot)+\beta_i^2(y,\cdot),\qquad y\in\Gamma_i.
\end{equation*}
The measures $\delta_{is}(y,\cdot)$ correspond to nonlocal terms
supported near the set $\cK$ of the conjugation points. The
measures $\alpha_i(y,\cdot)$ correspond to nonlocal terms
supported outside the set $\cK$. The measures $\beta_i^1(y,\cdot)$
and $\beta_i^2(y,\cdot)$ correspond to nonlocal terms with
arbitrary geometrical structure of their support (in particular,
their support may intersect with the set $\cK$); however, the
measure $\beta_i^1(y,\cM_p)$ of the set $\cM_p$ must be small for
small $p$ (Condition~\ref{cond2.7}) and the measure
$\beta_i^2(y,\cdot)$ must generate  a compact operator
(Condition~\ref{cond2.8}).

\begin{lemma}\label{l2.Boundary}
Let Conditions $\ref{cond1.1}$--$\ref{cond1.2}$,
$\ref{cond2.1'}$--$\ref{cond2.4}$, and
$\ref{cond2.5}$--$\ref{cond2.8}$ hold. Then there exists a bounded
operator $(\bI-\bB_{\alpha\beta}\bS_q)^{-1}:\cC_\cN(\pG)\to
\cC_\cN(\pG)$, $q\ge q_1$, where $q_1>0$ is sufficiently large.
\end{lemma}
\begin{proof}
1. Consider the bounded operators $
\hat\bB_{\beta}^1=\{\hat\bB_{\beta i}^1\}$,
$\tilde\bB_{\beta}^1=\{\tilde\bB_{\beta i}^1\}$,
$\bB_{\beta}^2=\{\bB_{\beta i}^2\}$, and
$\bB_{\alpha}=\{\bB_{\alpha i}\} $ acting from $C_\cN(\oG)$ to
$\cC_\cN(\pG)$ (cf.~\eqref{eqBAlphaBeta}).

Let us prove that the operator
$\bI-\bB_{\alpha}\bS_q:\cC_\cN(\pG)\to \cC_\cN(\pG)$ has the
bounded inverse.

 Introduce a function $\zeta\in C^\infty(\oG)$ such that
$0\le\zeta(y)\le 1$, $\zeta(y)=1$ for $y\in\overline{G_\sigma}$,
and $\zeta(y)=0$ for $y\notin G_{\sigma/2}$, where $\sigma>0$ is
the number from Condition~\ref{cond2.3}.

We have
\begin{equation}\label{eq75}
\bI-\bB_{\alpha}\bS_q=\bI-\bB_{\alpha}(1-\zeta)\bS_q-\bB_{\alpha}\zeta\bS_q.
\end{equation}

1a. First, we show that the operator
$\bI-\bB_{\alpha}(1-\zeta)\bS_q$ has the bounded inverse. By Lemma
\ref{l2.2} and Theorem~\ref{th1-2},
\begin{equation}\label{eq76}
\|\bB_{\alpha}(1-\zeta)\bS_q\|\le c_1.
\end{equation}
Furthermore, $(1-\zeta)\bS_q\psi=0$ in $\overline{G_\sigma}$ for
any  $\psi\in \cC_\cN(\pG)$. Therefore, by Condition
\ref{cond2.3},
\begin{equation}\label{eq77}
\supp\bB_\alpha(1-\zeta)\bS_q\psi\subset\pG\cap\overline{\cO_{\varkappa_2}(\cK)}.
\end{equation}

Let us show that
\begin{equation}\label{eq78}
\|[\bB_{\alpha}(1-\zeta)\bS_q]^2\|\le \frac{c}{q},\qquad q\ge q_1,
\end{equation}
where $q_1>0$ is sufficiently large and $c>0$ does not depend on
$q$. Consecutively applying (I) Lemma~\ref{l2.2}, (II) Lemma
\ref{l5} and relation~\eqref{eq77}, and (III) Lemma~\ref{l2.2} and
Theorem~\ref{th1-2}, we obtain
\begin{align*}
\|\bB_{\alpha}(1-\zeta)\bS_q\,\bB_{\alpha}(1-\zeta)\bS_q\psi\|_{\cC_\cN(\pG)}\le
&
\|\bS_q\bB_{\alpha}(1-\zeta)\bS_q\psi\|_{C_\cN(\oG\setminus\cO_{\varkappa_1}(\cK))}\le\\
&\dfrac{c_3}{q}\|\bB_{\alpha}(1-\zeta)\bS_q\psi\|_{C_\cN(\pG\cap\overline{\cO_{\varkappa_2}(\cK)})}\le
\dfrac{c_3c_1}{q}\|\psi\|_{\cC_\cN(\pG)}.
\end{align*}
This yields~\eqref{eq78} with $c=c_3c_1$.

If $q\ge 2c $, then the operator
$\bI-[\bB_{\alpha}(1-\zeta)\bS_q]^2$ has the bounded inverse.
Therefore, the operator $\bI-\bB_{\alpha}(1-\zeta)\bS_q$ also has
the bounded inverse and
\begin{equation}\label{eq79}
[\bI-\bB_{\alpha}(1-\zeta)\bS_q]^{-1}=[\bI+\bB_{\alpha}(1-\zeta)\bS_q]
[\bI-(\bB_{\alpha}(1-\zeta)\bS_q)^2]^{-1}.
\end{equation}
Representation~\eqref{eq79}, Lemma~\ref{l2.2}, Theorem~\ref{th1-2}
and relations~\eqref{eq76} and~\eqref{eq78} imply that
\begin{equation}\label{eq80}
\|[\bI-\bB_{\alpha}(1-\zeta)\bS_q]^{-1}\|=1+c_1+O(q^{-1}),\qquad
q\to+\infty.
\end{equation}

1b. Now we estimate the norm of the operator
$\bB_{\alpha}\zeta\bS_q$. Lemmas~\ref{l2.2} and~\ref{l5} imply
that
\begin{equation}\label{eq81}
\|\bB_{\alpha}\zeta\bS_q\psi\|_{\cC_\cN(\pG)}\le
\|\bS_q\psi\|_{C(\overline{G_{\sigma/2}})}\le
\dfrac{c_2}{q}\|\psi\|_{\cC_\cN(\pG)}.
\end{equation}
Therefore, using representation~\eqref{eq75}, we see that the
operator $\bI-\bB_{\alpha}\bS_q$ has the bounded inverse for
sufficiently large $q$ and
\begin{equation}\label{eq82}
(\bI-\bB_{\alpha}\bS_q)^{-1}=[\bI-(\bI-\bB_{\alpha}(1-\zeta)\bS_q)^{-1}\bB_\alpha\zeta
\bS_q]^{-1} [\bI-\bB_{\alpha}(1-\zeta)\bS_q]^{-1}.
\end{equation}
It follows from~\eqref{eq80}--\eqref{eq82} that
\begin{equation}\label{eq83}
\|(\bI-\bB_{\alpha}\bS_q)^{-1}\|=1+c_1+O(q^{-1}),\qquad
q\to+\infty.
\end{equation}

2. Let us prove that the operator
$\bI-(\bB_{\alpha}+\hat\bB_{\beta}^1+\tilde\bB_{\beta}^1)\bS_q:\cC_\cN(\pG)\to
\cC_\cN(\pG)$ has the bounded inverse.

2a. It follows from the definition of the operator
$\tilde\bB_{\beta}^1$ and from Lemma~\ref{l4} (with
$Q_1=\overline\cM$ and $Q_2=\oG\setminus\cM_{p/2}$) that
\begin{equation}\label{eq84}
\|\tilde\bB_{\beta
i}^1\bS_q\psi\|_{C_\cN(\overline{\Gamma_i})}\le\|\bS_q\psi\|_{C(\oG\setminus\cM_{p/2})}
\le\dfrac{c_2}{q}\|\psi\|_{\cC_\cN(\pG)}
\end{equation}
because $(\oG\setminus\cM_{p/2})\cap\overline\cM=\varnothing$ and
$\supp(\bS_q\psi)|_{\pG}\subset\overline\cM$ for
$\psi\in\cC_\cN(\pG)$.

2b. Let $\alpha_j(y,\oG)\ne 0$ for some $j$ and $y\in\Gamma_j$.
Due to Condition~\ref{cond2.7} (part 2) and Theorem~\ref{th1-2},
there is a number $d$ such that $0<2d<1/(1+c_1)$ and
\begin{equation}\label{eq85}
\|\hat\bB_{\beta i}^1\bS_q\psi\|_{C_\cN(\overline{\Gamma_i})}\le
\Bigg(\dfrac{1}{c_1(1+c_1)}-\dfrac{2d}{c_1}\Bigg)\|\bS_q\psi\|_{C_\cN(\oG)}\le
\Bigg(\dfrac{1}{1+c_1}-2d\Bigg)\|\psi\|_{\cC_\cN(\pG)}.
\end{equation}
Inequalities~\eqref{eq84} and~\eqref{eq85} yield
\begin{equation}\label{eq86}
\|(\hat\bB_{\beta}^1+\tilde\bB_{\beta}^1)\bS_q\|\le
\dfrac{1}{1+c_1}-d
\end{equation}
for sufficiently large $q$. Now it follows from~\eqref{eq83} and
\eqref{eq86} that $
\|(\bI-\bB_{\alpha}\bS_q)^{-1}(\hat\bB_{\beta}^1+\tilde\bB_{\beta}^1)\bS_q\|<1
$ for sufficiently large $q$. Hence, there exists the bounded
inverse operator
\begin{equation}\label{eq87}
[\bI-(\bB_\alpha+\hat\bB_{\beta}^1+\tilde\bB_{\beta}^1)\bS_q]^{-1}=
[\bI-(\bI-\bB_{\alpha}\bS_q)^{-1}(\hat\bB_{\beta}^1+\tilde\bB_{\beta}^1)\bS_q]^{-1}
[\bI-\bB_\alpha\bS_q]^{-1}.
\end{equation}

2c. If $\alpha_j(y,\oG)=0$ for $y\in\Gamma_j$, $j=1,\dots,N$,
then, due to Condition~\ref{cond2.7} (part 1), inequality
\eqref{eq85} assumes the form
\begin{equation*}
\|\hat\bB_{\beta i}^1\bS_q\psi\|_{C_\cN(\overline{\Gamma_i})}\le
\Bigg(\dfrac{1}{c_1}-\dfrac{2d}{c_1}\Bigg)\|\bS_q\psi\|_{C_\cN(\oG)}\le
(1-2d)\|\psi\|_{\cC_\cN(\pG)}.
\end{equation*}
Therefore, inequality~\eqref{eq86} reduces to
\begin{equation}\label{eq88}
\|(\hat\bB_{\beta}^1+\tilde\bB_{\beta}^1)\bS_q\|\le 1-d.
\end{equation}

Since $\bB_\alpha=0$ in the case under consideration, it follows
from~\eqref{eq88} that the operator
$$
\bI-(\bB_\alpha+\hat\bB_{\beta}^1+\tilde\bB_{\beta}^1)\bS_q=
\bI-(\hat\bB_{\beta}^1+\tilde\bB_{\beta}^1)\bS_q
$$
has the bounded inverse.

3. It remains to show that the operator
$\bI-\bB_{\alpha\beta}\bS_q$ also has the bounded inverse. By
Condition~\ref{cond2.8}, the operator $\bB_\beta^2$ is compact.
Therefore, the operator $\bB_\beta^2\bS_q$ is also compact. Since
the index of a Fredholm operator is stable under compact
perturbation, we see that the operator
$\bI-\bB_{\alpha\beta}\bS_q$ has the Fredholm property and
$\ind(\bI-\bB_{\alpha\beta}\bS_q)=0$. To prove that
$\bI-\bB_{\alpha\beta}\bS_q$ has the bounded inverse, it now
suffices to show that $\dim\ker(\bI-\bB_{\alpha\beta}\bS_q)=0$.

Let $\psi\in \cC_\cN(\pG)$ and
$(\bI-\bB_{\alpha\beta}\bS_q)\psi=0$. Then the function
$u=\bS_q\psi\in C^\infty(G)\cap C_\cN(\oG)$ is a solution of the
problem
\begin{gather*}
P_0u-qu=0,\quad y\in G,\\
u(y)-\bB_i u(y)-\bB_{\alpha i}u(y)-\bB_{\beta i}u(y) =0,\
y\in\Gamma_i;\qquad u(y)=0,\  y\in\cK.
\end{gather*}
By Corollary~\ref{cor2.1}, we have $u=0$. Therefore,
$\psi=\bB_{\alpha\beta}\bS_q\psi= \bB_{\alpha\beta}u=0$.
\end{proof}

\section{Existence of Feller Semigroups}\label{subsectBoundedExistence}
In this section, we prove that the above bounded perturbations of
elliptic equations with nonlocal conditions satisfying hypotheses
of Secs.~\ref{subsectStatement}--\ref{subsectBoundedReduction} are
generators of some Feller semigroups.

Reducing nonlocal problems to the boundary and using
Lemma~\ref{l2.Boundary}, we prove that the nonlocal problems are
solvable in the space of continuous functions.

\begin{lemma}\label{l2.4}
Let Conditions $\ref{cond1.1}$--$\ref{cond1.2}$,
$\ref{cond2.1''}$--$\ref{cond2.4}$, and
$\ref{cond2.5}$--$\ref{cond2.8}$ hold, and let $q_1$ be
sufficiently large. Then, for any $q\ge q_1$ and $f_0\in C(\oG)$,
 the problem
\begin{equation}\label{eql2.4_1}
qu(y)-P_0u(y)=f_0(y),\quad y\in G,
\end{equation}
\begin{equation}\label{eql2.4_2}
u(y)-\bB_i u(y)-\bB_{\alpha i}u(y)-\bB_{\beta i}u(y)=0,\
y\in\Gamma_i;\qquad u(y)=0,\  y\in\cK,
\end{equation}
admits a unique solution $u\in C_B(\oG)\cap W_{2,\loc}^{2}(G)$.
\end{lemma}
\begin{proof}
Let us consider the auxiliary problem
\begin{equation}\label{eq70}
qv(y)-P_0v(y)=f_0(y),\  y\in G;\qquad v(y)-\bB_i v(y) =0,\
y\in\Gamma_i,\ i=1,\dots,N.
\end{equation}

Since $f_0\in C(\oG)$, it follows from Theorem~\ref{th1-2} that
there exists a unique solution $v\in C_\cK(\oG)$  of problem
\eqref{eq70}. Therefore, $v\in C_\cN(\oG)$.

2. Set $w=u-v$. The unknown function $w$ belongs to $C_\cN(\oG)$,
and, by virtue of \eqref{eql2.4_1}--\eqref{eq70}, it satisfies the
relations
\begin{equation}\label{eq71}
\begin{aligned}
qw(y)-P_0w(y)&=0,& & y\in G,\\
w(y)-\bB_i w(y)-\bB_{\alpha i}w(y)-\bB_{\beta i}w(y)&=\bB_{\alpha
i}v(y)+\bB_{\beta i}v(y), & & y\in\Gamma_i,\
i=1,\dots,N,\\
w(y)&=0, & & y\in\cK.
\end{aligned}
\end{equation}

It follows from Condition~\ref{cond2.5} that problem~\eqref{eq71}
is equivalent to the operator equation
$\psi-\bB_{\alpha\beta}\bS_q\psi=\bB_{\alpha\beta}v$ for the
unknown function $\psi\in\cC_\cN(\pG)$. Lemma~\ref{l2.Boundary}
implies that this equation admits a unique solution
$\psi\in\cC_\cN(\pG)$. In this case, problem~\eqref{eql2.4_1},
\eqref{eql2.4_2} admits a unique solution
$$
u=v+w=v+\bS_q\psi=v+\bS_q(\bI-\bB_{\alpha\beta}\bS_q)^{-1}\bB_{\alpha\beta}v\in
C_B(\oG).
$$
Moreover, $u\in W^2_{2,\loc}(G)$ due to the interior regularity
theorem for elliptic equations.
\end{proof}

Using Lemma~\ref{l2.4} and the assumptions concerning the bounded
perturbations (see Condition~\ref{cond2.1'}), we prove that the
perturbed problems are solvable in the space of continuous
functions.

\begin{lemma}\label{l2.5}
Let Conditions $\ref{cond1.1}$--$\ref{cond1.2}$,
$\ref{cond2.1'}$--$\ref{cond2.4}$, and
$\ref{cond2.5}$--$\ref{cond2.8}$ hold, and let $q_1$ be
sufficiently large. Then, for any $q\ge q_1$ and $f_0\in C(\oG)$,
the problem
\begin{equation}\label{eql2.5_1}
qu-(P_0+P_1)u=f_0(y),\qquad y\in G,
\end{equation}
\begin{equation}\label{eql2.5_2}
u(y)-\bB_i u(y)-\bB_{\alpha i}u(y)-\bB_{\beta i}u(y) =0,\
y\in\Gamma_i;\qquad u(y) =0,\  y\in\cK,
\end{equation}
admits a unique solution $u\in C_B(\oG)\cap W_{2,\loc}^{2}(G)$.
\end{lemma}
\begin{proof}
Consider the operator $qI-P_0$ as the operator acting from
$C(\oG)$ to $C(\oG)$ with the domain
$$
\Dom(qI-P_0)=\{u\in C_B(\oG)\cap W^2_{2,\loc}(G): P_0u\in
C(\oG)\}.
$$
Lemma~\ref{l2.4} and Corollary~\ref{cor2.1} imply that there
exists the bounded operator $(qI-P_0)^{-1}: C(\oG)\to C(\oG)$ and
$$
\|(qI-P_0)^{-1}\|\le 1/q.
$$

Introduce the operator $qI-P_0-P_1: C(\oG)\to C(\oG)$ with the
domain $\Dom(qI-P_0-P_1)=\Dom(qI-P_0)$. Since
$$
qI-P_0-P_1= (I-P_1(qI-P_0)^{-1})(qI-P_0),
$$
it follows that the operator $qI-P_0-P_1: C(\oG)\to C(\oG)$ has
the bounded inverse for $q\ge q_1$, provided that $q_1$ is so
large that $ \|P_1\|\cdot \|(qI-P_0)^{-1}\|\le 1/2$, $ q\ge q_1.$
\end{proof}

We consider the unbounded operator $\bP_B: \Dom(\bP_B)\subset
C_B(\overline G)\to C_B(\overline G)$ given by
\begin{equation}\label{eqbP_BBoundedPert}
\bP_B u=P_0u+P_1u,\qquad u\in \Dom(\bP_B)=\{u\in C_B(\oG)\cap
W^2_{2,\loc}(G): P_0u+P_1u\in C_B(\overline G)\}.
\end{equation}

\begin{lemma}\label{l2.6}
Let Conditions $\ref{cond1.1}$--$\ref{cond1.2}$,
$\ref{cond2.1'}$--$\ref{cond2.4}$, and
$\ref{cond2.5}$--$\ref{cond2.8}$ hold. Then the set $\Dom(\bP_B)$
is dense in $C_B(\oG)$.
\end{lemma}
\begin{proof}
We will follow the scheme proposed in~\cite{GalSkJDE}.

1. Let $u\in C_B(\oG)$. Since $C_B(\oG)\subset C_\cN(\oG)$ due to
\eqref{eqBNK}, it follows that, for any $\varepsilon>0$ and $q\ge
q_1$, there is a function $u_1\in C^\infty(\oG)\cap C_\cN(\oG)$
such that
\begin{equation}\label{eq3.12}
\|u-u_1\|_{C(\oG)}\le\min(\varepsilon,\varepsilon/(2c_1k_q)),
\end{equation}
where $k_q=\|(\bI-\bB_{\alpha\beta}\bS_q)^{-1}\|$.

Set
\begin{equation}\label{eq3.14}
\begin{aligned}
f_0(y)&\equiv qu_1-P_0 u_1, & & y\in G,\\
\psi_i(y)&\equiv u_1(y)-\bB_i u_1(y)-\bB_{\alpha
i}u_1(y)-\bB_{\beta i}u_1(y),& & y\in\Gamma_i,\ i=1,\dots,N.
\end{aligned}
\end{equation}

Since $u_1\in C_\cN(\oG)$, it follows from Condition~\ref{cond2.5}
that $\{\psi_i\}\in \cC_\cN(\pG)$. Using the relation
$$
u(y)-\bB_i u(y)-\bB_{\alpha i}u(y)-\bB_{\beta i}u(y)=0,\qquad
 y\in\Gamma_i,
$$
inequality~\eqref{eq3.12}, and Lemma~\ref{l2.2}, we obtain
\begin{equation}\label{eq3.13}
\|\{\psi_i\}\|_{\cC_\cN(\pG)}\le\|u-u_1\|_{C(\oG)}
+\|(\bB+\bB_{\alpha\beta})(u-u_1)\|_{\cC_\cN(\pG)}\le\varepsilon/(c_1k_q).
\end{equation}

Consider the auxiliary nonlocal problem
\begin{equation}\label{eq3.15}
\begin{gathered}
qu_2-P_0 u_2 = f_0(y),\quad y\in G,\\
u_2(y)-\bB_i u_2(y)-\bB_{\alpha i}u_2(y)-\bB_{\beta i}u_2(y) =0,\
y\in\Gamma_i;\qquad u_2(y) =0,\ y\in\cK.
\end{gathered}
\end{equation}
Since $f_0\in C^\infty(\oG)$, it follows from Lemma~\ref{l2.4}
that problem~\eqref{eq3.15} has a unique solution $u_2\in
C_B(\oG)\subset C_\cN(\oG)$.

Using~\eqref{eq3.14},~\eqref{eq3.15}, and the relations
$u_1(y)=u_2(y)=0$, $y\in\cK$, we see that the function
$w_1=u_1-u_2$ satisfies the relations
\begin{equation}\label{eq3.16}
\begin{gathered}
q w_1-P_0 w_1 =0,\quad y\in G,\\
w_1(y)-\bB_i w_1(y)-\bB_{\alpha i}w_1(y)-\bB_{\beta i}w_1(y)
=\psi_i(y),\ y\in\Gamma_i;\qquad w_1(y) =0,\ y\in\cK.
\end{gathered}
\end{equation}
It follows from Condition~\ref{cond2.5} that problem
\eqref{eq3.16} is equivalent to the operator equation
$\phi-\bB_{\alpha\beta}\bS_q\phi=\psi$ in $\cC_\cN(\pG)$, where
$w_1=\bS_q\phi$. Lemma~\ref{l2.Boundary} implies that this
equation admits a unique solution $\phi\in\cC_\cN(\pG)$.
Therefore, using Theorem~\ref{th1-2} and
inequality~\eqref{eq3.13}, we obtain
\begin{equation}\label{eq3.17}
\|w_1\|_{C(\oG)}\le c_1\|(\bI-\bB_{\alpha\beta}\bS_q)^{-1}\|\cdot
\|\{\psi_i\}\|_{\cC_\cN(\pG)}\le c_1
k_q\varepsilon/(c_1k_q)=\varepsilon.
\end{equation}

2. Finally, we consider the problem
\begin{equation}\label{eq3.18}
\begin{gathered}
\lambda u_3-P_0u_3-P_1u_3 =\lambda u_2,\quad y\in G,\\
u_3(y)-\bB_i u_3(y)-\bB_{\alpha i}u_3(y)-\bB_{\beta i}u_3(y) =0,\
y\in\Gamma_i;\qquad u_3(y) =0,\ y\in\cK.
\end{gathered}
\end{equation}
Since $u_2\in C_B(\oG)$, it follows from Lemma~\ref{l2.5} that
problem~\eqref{eq3.18} admits a unique solution $u_3\in
\Dom(\bP_B)$ for sufficiently large $\lambda$.

Denote $w_2=u_2-u_3$. It follows from~\eqref{eq3.18} that
$$
\lambda w_2-P_0w_2-P_1w_2=-P_0 u_2-P_1u_2=f_0-qu_2-P_1u_2.
$$
Applying Corollary~\ref{cor2.1}, we have
$$
\|w_2\|_{C(\oG)}\le\dfrac{1}{\lambda}\|f_0-qu_2-P_1u_2\|_{C(\oG)}.
$$
Choosing sufficiently large $\lambda$ yields
\begin{equation}\label{eq3.20}
\|w_2\|_{C(\oG)}\le\varepsilon.
\end{equation}

Inequalities~\eqref{eq3.12},~\eqref{eq3.17}, and~\eqref{eq3.20}
imply
$$
\|u-u_3\|_{C(\oG)}\le\|u-u_1\|_{C(\oG)}+\|u_1-u_2\|_{C(\oG)}+
\|u_2-u_3\|_{C(\oG)}\le 3\varepsilon.
$$
\end{proof}

Now we can prove the  main result of the paper.
\begin{theorem}\label{th2.1}
Let Conditions $\ref{cond1.1}$--$\ref{cond1.2}$,
$\ref{cond2.1'}$--$\ref{cond2.4}$, and
$\ref{cond2.5}$--$\ref{cond2.8}$ hold. Then the operator
$\bP_B:\Dom(\bP_B)\subset C_B(\oG)\to C_B(\oG)$ is a generator of
a Feller semigroup.
\end{theorem}
\begin{proof}
1. By Lemma~\ref{l2.5} and Corollary~\ref{cor2.1}, there exists
the bounded operator $(qI-\bP_B)^{-1}: C_B(\oG)\to C_B(\oG)$ and
$$
\|(qI-\bP_B)^{-1}\|\le 1/q
$$
for all sufficiently large $q>0$.

2. Since the operator $(qI-\bP_B)^{-1}$ is bounded and defined on
the whole space $C_B(\oG)$, it is closed. Therefore, the operator
$qI-\bP_B:\Dom(\bP_B)\subset C_B(\oG)\to C_B(\oG)$ is closed.
Hence, $\bP_B:\Dom(\bP_B)\subset C_B(\oG)\to C_B(\oG)$ is also
closed.

3. Let us prove that the operator $(qI-\bP_B)^{-1}$ is
nonnegative. Assume the contrary; then there exists a function
$f_0\ge0$ such that a solution $u\in\Dom(\bP_B)$ of the equation
$qu-\bP_Bu=f_0$ achieves its negative minimum at some point
$y^0\in \oG$. In this case, the function $v=-u$ achieves its
positive maximum at the point $y^0$. By Lemma~\ref{l2.3}, there is
a point $y^1\in G$ such that $v(y^1)=v(y^0)$ and $\bP_B v(y^1)\le
0$. Therefore, $ 0<v(y^0)=v(y^1)=(\bP_Bv(y^1)-f_0(y^1))/q\le 0. $
This contradiction proves that $u\ge0$.

Thus, all the hypotheses of the Hille--Iosida theorem
(Theorem~\ref{thHI}) are fulfilled. Hence,
$\bP_B:\Dom(\bP_B)\subset C_B(\oG)\to C_B(\oG)$ is a generator of
a Feller semigroup.
\end{proof}

As a conclusion, we give an example of nonlocal conditions
satisfying the assumptions of the paper.

Let $G\subset\bbR^2$ be a bounded domain with boundary
$\pG=\Gamma_1\cup\Gamma_2\cup\cK$, where $\Gamma_1$ and $\Gamma_2$
are $C^\infty$ curves open and connected in the topology of $\pG$
such that $\Gamma_1\cap\Gamma_2=\varnothing$ and
$\overline{\Gamma_1}\cap\overline{\Gamma_2}=\cK$; the set $\cK$
consists of two points $g_1$ and $g_2$. We assume that the domain
$G$ coincides with some plane angle in an
$\varepsilon$-neighborhood of the point $g_i$, $i=1,2$. Let
$\Omega_j$, $j=1,\dots,4$, be continuous transformations defined
on $\overline{\Gamma_1}$ and satisfying  the following conditions
(see Fig.~\ref{figEx3-4}):
\begin{figure}[ht]
{ \hfill\epsfxsize130mm\epsfbox{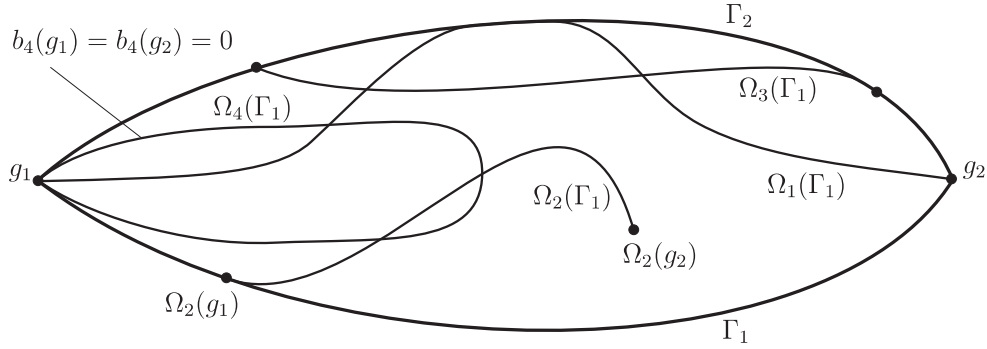}\hfill\ }
\caption{Nontransversal nonlocal conditions}
   \label{figEx3-4}
\end{figure}
\begin{enumerate}
\item
$\Omega_1(\cK)\subset\cK$,
$\Omega_1(\Gamma_1\cap\cO_\varepsilon(\cK))\subset G$,
$\Omega_1(\Gamma_1\setminus\cO_\varepsilon(\cK))\subset
G\cup\Gamma_2$, and $\Omega_1(y)$ is a composition of shift of the
argument, rotation, and homothety for
$y\in\overline{\Gamma_1}\cap\cO_\varepsilon(\cK)$;
\item
there exist numbers $\varkappa_1>\varkappa_2>0$ and $\sigma>0$
such that
$\Omega_2(\overline{\Gamma_1})\subset\oG\setminus\cO_{\varkappa_1}(\cK)$
and
$\Omega_2(\overline{\Gamma_1}\setminus\cO_{\varkappa_2}(\cK))\subset
\overline{G_\sigma}$; moreover, $\Omega_2(g_1)\in\Gamma_1$ and
$\Omega_2(g_2)\in G$;
\item
$\Omega_3(\overline{\Gamma_1})\subset G\cup{\Gamma_2}$ and
$\Omega_3(\cK)\subset\Gamma_2$;
\item
$\Omega_4(\overline{\Gamma_1})\subset G\cup\overline{\Gamma_2}$
and $\Omega_4(\cK)\subset\cK$.
\end{enumerate}

Let  $b_1\in C(\overline{\Gamma_1})\cap
C^\infty(\overline{\Gamma_1}\cap\cO_\varepsilon(\cK))$,
$b_2,b_3,b_4\in C(\overline{\Gamma_1})$, and $b_j\ge0$,
$j=1,\dots,4$.

Let $G_1$ be a bounded domain, $G_1\subset G$, and $\Gamma\subset
\oG$ be a curve of class $C^1$. Introduce continuous nonnegative
functions $c(y,\eta)$, $y\in\overline{\Gamma_1}$,
$\eta\in\overline{G_1}$, and $d(y,\eta)$,
$y\in\overline{\Gamma_1}$, $\eta\in\overline{\Gamma}$.

Consider the following nonlocal conditions:
\begin{equation}\label{eq89-90}
\begin{aligned}
u(y)-\sum\limits_{j=1}^4 b_j(y)u(\Omega_j(y))-
\int\limits_{G_1}c(y,\eta)u(\eta)d\eta-
\int\limits_{\Gamma}d(y,\eta)u(\eta)d\Gamma_\eta &=0, &
&y\in\Gamma_1,\\
u(y) &=0, & &y\in\overline{\Gamma_2}.
\end{aligned}
\end{equation}

Let $Q\subset\oG$ be an arbitrary Borel set; introduce the measure
$\mu(y,\cdot)$, $y\in\pG$:
\begin{equation*}
\begin{aligned}
\mu(y,Q)&=\sum\limits_{j=1}^4 b_j(y)\chi_Q(\Omega_j(y))+
\int\limits_{G_1\cap Q}c(y,\eta)d\eta+ \int\limits_{\Gamma\cap
Q}d(y,\eta)u(\eta)d\Gamma_\eta, &
&y\in\Gamma_1,\\
\mu(y,Q) &=0, & &y\in\overline{\Gamma_2},
\end{aligned}
\end{equation*}

Let $\cN$ and $\cM$ be defined as before. Assume that
\begin{equation*}
\begin{gathered}
\mu(y,\oG)=\sum\limits_{j=1}^4 b_j(y)+\int\limits_{G_1}
c(y,\eta)\,d\eta+\int\limits_\Gamma d(y,\eta)\,d\Gamma_\eta\le
1,\quad y\in\pG,\\
\int\limits_{\Gamma\cap\cM}d(y,\eta)d\Gamma_\eta<1,\quad y\in\cM;\\
 b_2(g_1)=0\
\text{or}\ \mu(\Omega_2(g_1),\oG)=0,\quad b_2(g_2)=0;\quad
b_4(g_j)=0;\quad c(g_j,\cdot)=0;\quad d(g_j,\cdot)=0.
\end{gathered}
\end{equation*}

Setting $b(y)=1-\mu(y,\oG)$, we can rewrite \eqref{eq89-90} in the
form (cf.~\eqref{eq56})
$$
b(y)u(y)+\int\limits_\oG [u(y)-u(\eta)]\mu(y,d\eta)=0,\quad
y\in\pG.
$$

Introduce a cut-off function $\zeta\in C^\infty(\bbR^2)$ supported
in $\cO_\varepsilon(\cK)$, equal to $1$ on
$\cO_{\varepsilon/2}(\cK)$, and such that $0\le\zeta(y)\le 1$ for
$y\in\bbR^2$.  Let $y\in\overline{\Gamma_1}$ and $Q\subset\oG$ be
a Borel set; denote
\begin{equation*}
\begin{gathered}
\delta(y,Q)=\zeta(y)b_1(y)\chi_Q(\Omega_1(y)),\qquad
\alpha(y,Q)=b_2(y)\chi_Q(\Omega_2(y)),\\
\beta^1(y,Q)=\big(1-\zeta(y)\big)b_1(y)\chi_Q(\Omega_1(y))+\sum\limits_{j=3,4}
b_j(y)\chi_Q(\Omega_j(y)),\\
 \beta^2(y,Q)=\int\limits_{G_1\cap
Q}c(y,\eta)d\eta+ \int\limits_{\Gamma\cap Q}d(y,\eta)u(\eta)d\Gamma_\eta
\end{gathered}
\end{equation*}
(for simplicity, we have omitted the subscript ``1'' in the notation of the measures).
One can directly verify that these measures satisfy
Conditions~\ref{condK1},~\ref{cond1.2}, \ref{cond2.1''}--\ref{cond2.4},
and~\ref{cond2.5}--\ref{cond2.8}.

\smallskip

The author is grateful to Prof. A.L. Skubachevskii for attention
to this work.

\end{document}